\documentclass[11pt]{article} 
\usepackage{a4,amssymb, amsmath, amsthm}

\newtheorem{thm}{Theorem}[section]
\newtheorem{cor}[thm]{Corollary}
\newtheorem{lem}[thm]{Lemma}

\theoremstyle{definition}

\theoremstyle{remark}
\newtheorem{rem}[thm]{Remark}

\newcommand{\A}{\mathcal{A}}

\newcommand{\R}{\mathbb{R}}

\newcommand{\om}{\Omega }
\newcommand{\pom}{{\partial \Omega} }
\newcommand{\omc}{ {\Omega^c} }
\newcommand{\gs}{{\Gamma_s}}
\newcommand{\gt}{{\Gamma_t}}
\newcommand{\Wgs}{{\widetilde{W}^{\frac{1}{2},2}(\gs)}}
\newcommand{\Whgs}{{\widetilde{W}_h^{\frac{1}{2},2}(\gs)}}
\newcommand{\Wg}{{W^{\frac{1}{2},2}(\pom)}}
\newcommand{\Whg}{{W_h^{\frac{1}{2},2}(\pom)}}
\newcommand{\Wmg}{{W^{-\frac{1}{2},2}(\pom)}}
\newcommand{\Whmg}{{W_h^{-\frac{1}{2},2}(\pom)}}
\newcommand{\W}{{W^{1,4}(\om)}}
\newcommand{\Wh}{{W_h^{1,4}(\om)}}

\begin{document}

\title{FE--BE coupling for a Transmission Problem Involving Microstructure\\\vskip 0.8cm}
\author{H.~Gimperlein, M.~Maischak, E.~Schrohe, E.~P.~Stephan}
\date{}

\maketitle \vskip 0.5cm
\begin{abstract}
\noindent We analyze a finite element/boundary element procedure to
solve a non-convex contact problem for the double--well potential.
After relaxing the associated functional, the degenerate
minimization problem is reduced to a boundary/domain variational
inequality, a discretized saddle point formulation of which may then
be solved numerically. The convergence of the Galerkin
approximations to certain macroscopic quantities and a corresponding
a posteriori estimate for the approximation error are discussed.
\end{abstract}

\vskip 1.0cm

\section{Introduction}\label{intro}

Adaptive finite element / boundary element procedures provide an
efficient and extensively investigated tool for the numerical
solution of uniformly elliptic transmission or contact problems.
However, models of strongly nonlinear materials often lead to
nonelliptic partial differential equations, where the standard
Hilbert space techniques are no longer appropriate to analyze the
computational methods. In a previous work \cite{scalar} we showed
that certain mixed $L^2-L^p$--Sobolev spaces provide a convenient
setting to study contact problems for monotone operators like the
$p$--Laplacian. This article extends the approach to nonconvex
functionals, discussing the prototypical model case of a
double--well potential in Signiorini and transmission contact with
the linear Laplace equation. As a proof of principle, it intends to
clarify the mathematical basis -- including well--posedness,
convergence, a priori and a simple a posteriori estimate -- of
adaptive finite element / boundary element methods in this highly
degenerate nonlinear setting. The methods readily extend to certain
systems of equations from nonlinear elasticity, frictional contact
or more elaborate a posteriori estimates.

Let $\Omega \subset \R^n$ be a bounded Lipschitz domain and $\pom =
\overline{\gt \cup \gs}$ a decomposition of its boundary into
disjoint open subsets, $\gt \neq \emptyset$. We consider the problem
of minimizing the functional
$$\Phi(u_1, u_2) = \int_\om W(\nabla u_1) +  \frac{1}{2} \int_\omc |\nabla u_2|^2
-\int_\om f u_1 - \langle t_0, u_2|_\pom\rangle$$ with nonconvex
energy density $W(F) = |F-F_1|^2\ |F-F_2|^2$ ($F_1 \neq F_2 \in
\R^n$) over the closed convex set
$$\{(u_1, u_2) \in W^{1,4}(\om) \times W^{1,2}_{loc}(\omc) : (u_1-u_2)|_\gt= u_0,\, (u_1-u_2)|_\gs \leq u_0, \, u_2 \in \mathcal{L}_2\},$$
$$\mathcal{L}_2 = \left\{\{v \in W^{1,2}_{loc}(\omc): \Delta v = 0\ \text{ in $W^{-1,2}(\omc)$,\  $v=\left\{\begin{array}{l@{}l}
      o(1)&, n=2 \\
     \mathcal{O}(|x|^{2-n})\ &, n>2
      \end{array}\right\}$}\right\}.$$
The data $f \in L^{4/3}(\om)$, $t_0 \in W^{-\frac{1}{2},2}(\pom)$
and $u_0 \in W^{\frac{1}{2},2}(\pom)$ are taken from the appropriate
spaces.

Classical exact minimizers of $\Phi$ satisfy the equations
\begin{align}
-\mathrm{div}\ DW(\nabla u_1)= f \quad \text{in $\Omega$,}\quad
\Delta u_2 =&0 \quad \text{in $\Omega^c$,} \nonumber\\
\nu \cdot DW(\nabla u_1)- \partial_\nu u_2=t_0 \quad \text{on
$\partial \Omega$,}\,\,
u_1-u_2= &u_0\,\, \text{on $\Gamma_t$,} \nonumber\\
u_1-u_2\leq u_0,\,\, \nu \cdot DW(\nabla u_1)\leq 0, \,\,\nu \cdot
DW(\nabla u_1) (u_1-u_2-u_0) =& 0 \quad
\text{on $\Gamma_s$,} \nonumber \\
\text{$+$ radiation condition for $u_2$ at $\infty$}.\nonumber &
\end{align}
Therefore, the minimization problem for $\Phi$ is a variational
formulation of a contact problem between the double--well potential
$W$ and the Laplace equation, with transmission ($\gt$) and
Signiorini ($\gs$) contact at the interface.

Nonconvex minimization problems of this type arise naturally when a
material in $\Omega$ passes the critical point of a phase transition
into a finely textured mixture of locally energetically equivalent
configurations of lower symmetry, the so--called microstructure.
Lacking convexity in $\om$, the minimum of $\Phi$ is usually not
attained. Nevertheless, it is possible and of practical interest to
extract average physical properties of the sequences minimizing
$\Phi$. Examples of such quantities include the displacement in the
exterior, stresses, the region, where minimizing sequences develop
microstructure, or also the gradient of the displacement away from
the microstructure. Crucially for the use of boundary elements, the
exterior boundary value on the interface is not affected by the
presence of microstructure.

The increasingly fine length scale of the microstructure often
prevents the direct numerical minimization, and starting with works
of Carstensen and Plech\'a\v{c} \cite{cp1, cp2} computational
approaches based on \emph{relaxed} formulations have been
considered. Relaxation amounts to replacing the nonconvex functional
by its quasi--convex envelope, in our setting the degenerate
functional
$$\Phi^{**}(u_1, u_2) = \int_\om W^{**}(\nabla u_1) +  \frac{1}{2}
\int_\omc |\nabla u_2|^2 -\int_\om f u_1 - \langle t_0,
u_2|_\pom\rangle.$$ If $A = \frac{1}{2} (F_2 - F_1)$ and $B =
\frac{1}{2} (F_1 + F_2)$, the convex integrand $W^{**}$ is given by
the formula (cf.~\cite{cp1})
$$W^{**}(F) = \left(\max \{0,|F-B|^2-|A|^2\}\right)^2 + 4
|A|^2|F-B|^2-4 (A(F-B))^2.$$ The theory of relaxation for nonconvex
integrands shows that the weak limit of any $\Phi$--minimizing
sequence minimizes $\Phi^{**}$. Macroscopic quantities like the
stress $DW^{**}$ on $\Omega$ defined by this weak limit coincide
with the averages such as the average stress $\int DW(u)\
\mathrm{d}\nu(u)$ defined by the Young measure $\nu$ associated to
the minimizing sequence. To extract the average physical properties
of sequences minimizing $\Phi$, it is thus sufficient to understand
the minimizers of the degenerately convex functional $\Phi^{**}$.

We are thus going to analyze a finite element / boundary element
scheme which numerically minimizes $\Phi^{**}$ and thereby
approximates certain macroscopic quantities independent of the
particular minimizer. Our approach is based on previous works by
Carstensen / Plech\'a\v{c} \cite{cp1, cp2} and Bartels \cite{b} for
double--well potentials with Dirichlet or Neumann boundary
conditions. Section \ref{sec::relaxed} discusses the relaxed problem
and identifies several quantities shared by its minimizers. A priori
error estimates for their computation and convergence are
established in Section \ref{sec::apriori}. Section
\ref{sec::adaptive} contains an a posteriori estimate of residual
type, on which an adaptive grid refinement strategy may be based.
\\

For later reference, we recall from \cite{cp1} the following
estimates for the relaxed double--well potential ($E,F \in \R^n$):
\begin{equation} \label{W1}
\max\{C_1 |F|^4 - C_2, 0\} \leq W^{**}(F) \leq C_3 + C_4 |F|^4,
\end{equation}
\begin{equation} \label{W2}
|DW^{**}(F)| \leq C_5 (1+|F|^3),
\end{equation}
\begin{equation} \label{W3}
|DW^{**}(F) - DW^{**}(E)|^2 \leq C_6 (1+|F|^2+|E|^2)(DW^{**}(F) -
DW^{**}(E)) (F-E),
\end{equation}
\begin{align}
8 |A|^2 |\mathbb{P}F-\mathbb{P}E|^2 + 2 \frac{Q(F)+Q(E)}{|A|} |A
(F-E)|^2+2(Q(F)-Q(E))^2 \nonumber \\ \label{W4} \leq (DW^{**}(F) -
DW^{**}(E)) (F-E)
\end{align}
where $Q(F) = \max \{0,|F-B|^2-|A|^2\}$ and $\mathbb{P}$ is the
orthogonal projection onto the subspace of vectors orthogonal to
$A$.

\vspace*{0.2cm}
\section{Analysis of the Relaxed Problem}\label{sec::relaxed}

We first outline how the minimization problem for $\Phi^{**}$ can be
reduced to a boundary--domain variational inequality. As it involves
the exterior problem, it is not affected by the nonconvex part of
the functional. See e.g.~\cite{scalar} for a more detailed
exposition.

Recall the Steklov--Poincar\'e operator
$$S : W^{\frac{1}{2},2}(\pom) \to W^{-\frac{1}{2},2}(\pom),$$ a
positive and selfadjoint operator (pseudodifferential of order $1$,
if $\pom$ is smooth) with defining property
$$\partial_\nu u_2|_\pom = - S (u_2|_\pom)$$ for solutions $u_2 \in
\mathcal{L}_2$ of the Laplace equation on $\omc$. Let $$\Wgs = \{v
\in W^{\frac{1}{2},2}(\pom) : \mathrm{supp}\ v \subset
\bar{\Gamma}_s\}, \quad  X = W^{1,4}(\om) \times \Wgs.$$ Using $S$,
an affine change of variables,
$$(u_1,u_2) \mapsto (u,v) = (u_1 - c, u_0+u_2|_\pom-u_1|_\pom) \in X,$$
for a suitable $c \in \R$ reduces the exterior part of $\Phi^{**}$
to $\gs$:
$$\Phi^{**}(u_1,u_2) = \int_\om W^{**}(\nabla u) + \frac{1}{2} \langle S(u|_\pom +v),u|_\pom
+v\rangle - \lambda(u,v) + C \equiv J(u,v) + C,$$ where
$$\lambda(u,v) = \langle t_0 +S u_0, u|_\pom + v\rangle + \int_\om f
u$$ and $C=C(u_0,t_0)$ is a constant independent of $u,v$.
Therefore, instead of $\Phi^{**}$ one may equivalently minimize $J$
over
$$\mathcal{A} = \{(u,v) \in X : v \geq 0 \text{ and } \langle S(u|_\pom +v-u_0), 1\rangle
= 0 \,\, \text{if $n=2$}\}.$$ A reformulation as a variational
inequality reads as follows: Find $(\hat u, \hat v) \in \mathcal{A}$
such that
\begin{equation} \label{VI}
\int_\om DW^{**}(\nabla \hat u) \nabla(u-\hat u) + \langle S (\hat
u|_\pom +\hat v), (u-\hat u)|_\pom + v-\hat v\rangle \geq
\lambda(u-\hat u,v-\hat v)
\end{equation}
for all $(u,v) \in \mathcal{A}$.

Convexity and the closedness of $\mathcal{A}$ assure that the
relaxed functional $J$ assumes its minimum. Due to the lack of
coercivity, the minimizer may fail to be unique, though certain
macroscopic quantities are uniquely determined.
\begin{lem} \label{uniqueness}
The set of minimizers is nonempty and bounded in $X$. The stress
$DW^{**}(\hat u)$, the projected gradient $\mathbb{P} \nabla \hat
u$, the region of microstructure $\{x \in \om : Q(\nabla \hat u) =
0\}$ and the boundary value $\hat u|_\pom+\hat v$ are independent of
the minimizer $(\hat u,\hat v)\in \mathcal{A}$ of $J$ (up to sets of
measure $0$).
\end{lem}

For the proof, we recall the variant
\begin{equation}\label{friedrichsineq}
\|\hat u\|_{W^{1,4}(\om)}\lesssim \|\nabla\hat u\|_{L^4(\om)} +
\|\hat u \|_{W^{\frac{1}{2},2}(\gt)}
\end{equation}
of Friedrichs' inequality from \cite{scalar}.

\begin{proof}[Proof of Lemma \ref{uniqueness}]
By (\ref{W1}) and the coercivity of $S$, we have
\begin{eqnarray*}
|J(\hat u,\hat v)| &\geq& C_1 \|\nabla \hat u\|_{L^4(\om)}^4 -C_2
\mathrm{vol}\ \om
+ \frac{1}{2} C_S\|\hat u|_{\gs} +v\|_{W^{\frac{1}{2},2}(\gs)}^2\\
& & \quad+\ \frac{1}{2} C_S\| \hat u \|_{W^{\frac{1}{2},2}(\gt)}^2 -
\|f\|_{L^{4/3}(\om)} \|\hat u\|_{W^{1,4}(\om)} \\ & &\quad -\ \|t_0
+S u_0\|_{W^{-\frac{1}{2},2}(\pom)} \|\hat u|_{\gs} +\hat
v\|_{W^{\frac{1}{2},2}(\gs)} \\ & & \quad - \ \|t_0
+S u_0\|_{W^{-\frac{1}{2},2}(\pom)} \|\hat u\|_{W^{\frac{1}{2},2}(\gt)}\\
\end{eqnarray*}
for any minimizer $(\hat u,\hat v) \in \mathcal{A}$ of $J$.
Consequently
$$\|\nabla \hat u\|_{L^4(\om)}^4 + \|\hat u|_{\gs}
+\hat v\|_{W^{\frac{1}{2},2}(\gs)}^2 + \|\hat u
\|_{W^{\frac{1}{2},2}(\gt)}^2 - C \|\hat u\|_{W^{1,4}(\om)}$$ is
bounded for some $C>0$. The inequality (\ref{friedrichsineq}) easily
yields the boundedness of $\|(\hat u,\hat v)\|_X$.

If $(\hat u_1,\hat v_1), (\hat u_2,\hat v_2) \in \mathcal{A}$ are
two minimizers, $J$ is constant on $\{(\hat u_1,\hat v_1) + s(\hat
u_2-\hat u_1, \hat v_2-\hat v_1) : s \in [0,1]\}$: If not, the
restriction of $J$ to this set would have a maximum $> J(\hat
u_1,\hat v_1)=J(\hat u_2,\hat v_2)$ for some $0<s<1$, contradicting
the convexity of $J$. Therefore
$$\langle J'(\hat u_2,\hat v_2)-J'(\hat u_1,\hat v_1), (\hat u_2 - \hat u_1 ,\hat v_2- \hat v_1)\rangle=0,$$
or for our particular $J$
\begin{eqnarray*}
0 & = & \int_\om (DW^{**}(\nabla \hat u_2) - DW^{**}(\nabla \hat u_1)) \nabla(\hat u_2-\hat u_1) \\
&& \qquad +\ \langle S ((\hat u_2-\hat u_1)|_\pom + \hat v_2-\hat
v_1), (\hat u_2-\hat u_1)|_\pom + \hat v_2-\hat v_1\rangle.
\end{eqnarray*}
Both of the terms on the right hand side are non--negative, because
$S$ is coercive and $W^{**}$ convex, and hence $$\hat u_1|_\pom +
\hat v_1= \hat u_2|_\pom + \hat v_2 \ \text{ and }\ (DW^{**}(\nabla
\hat u_2) - DW^{**}(\nabla \hat u_1)) \nabla(\hat u_2-\hat u_1) =
0$$ almost everywhere. The inequality (\ref{W3}),
$$
|DW^{**}(\nabla \hat u_2) - DW^{**}(\nabla \hat u_1)|^2 \lesssim
(1+|\nabla \hat u_2|^2+|\nabla \hat u_1|^2)(DW^{**}(\nabla \hat u_2)
- DW^{**}(\nabla \hat u_1)) \nabla (\hat u_2-\hat u_1),$$ implies
$DW^{**}(\nabla \hat u_1) = DW^{**}(\nabla \hat u_2)$ almost
everywhere. The assertions about the projected gradients and the
region of microstructure are immediate consequences of inequality
(\ref{W4}),
\begin{align*}
|\mathbb{P}\nabla \hat u_2-\mathbb{P}\nabla \hat u_1|^2 + (Q(\nabla
\hat u_2)-Q(\nabla \hat u_1))^2  \lesssim (DW^{**}(\nabla \hat u_2)
- DW^{**}(\nabla \hat u_1)) \nabla (\hat u_2-\hat u_1).
\end{align*}
\end{proof}

In particular, the displacement $\hat u_2$ on $\omc$ is uniquely
determined and may be computed from $\hat u|_\pom+\hat v$ with the
help of layer potentials. Due to the lack of convexity of $W$,
neither $\hat u$ nor $\nabla \hat u$ needs to be unique. However,
Lemma \ref{uniqueness} allows to identify subsets of $\om$, on which
these quantities are well--defined.

\begin{cor}
a) Let $\om_{t,A}$ be the set of points $x \in \om$ for which the
component of a hyperplane perpendicular to $A$ through $x$
intersects $\gt$. Then
the displacement $u|_{\om_{t,A}}$ is independent of the minimizer.\\
b) The same holds for the gradient $\nabla \hat u$ outside the
region of microstructure.
\end{cor}
\begin{proof}
The proof closely follows the arguments of \cite{cp1}, Theorem
3.\\
a) Let $(\hat u_1,\hat v_1)$, $(\hat u_2, \hat v_2) \in \mathcal{A}$
be two minimizers, and consider $w=\hat u_2-\hat u_1$. Because
$\mathbb{P}\nabla \hat u_1 = \mathbb{P}\nabla \hat u_2$, $\nabla w$
is parallel to $A$ almost everywhere. It is easy to see that,
therefore, $w$ may be modified on a set of measure zero to yield an
absolutely continuous function which is locally constant along the
hyperplanes perpendicular to $A$. With $w|_{\gt}$ being $0$ by Lemma
\ref{uniqueness}, $w$ also has to
vanish on almost every hyperplane hitting $\gt$.\\
b) is a consequence of $\mathbb{P}\nabla \hat u_1 = \mathbb{P}\nabla
\hat u_2$, $DW^{**}(\nabla \hat u_1) = DW^{**}(\nabla \hat u_2)$ and
(\ref{W4}): $$(Q(\nabla \hat u_2)+Q(\nabla \hat u_1)) |A \nabla
(\hat u_2-\hat u_1)|^2 \lesssim (DW^{**}(\nabla \hat u_2) -
DW^{**}(\nabla \hat u_1)) \nabla (\hat u_2-\hat u_1).$$
\end{proof}

\vspace*{0.2cm}
\section{Discretization and A Priori Estimates}\label{sec::apriori}

We are now going to analyze which quantities can be computed
numerically with a Galerkin method.

Let $\{\mathcal{T}_h\}_{h\in I}$ a regular triangulation of
$\om\subset \R^2$ into disjoint open regular triangles $K$, so that
$\overline{\om} = \bigcup_{K \in \mathcal{T}_h} K$. Each element has
at most one edge on $\pom$, and the closures of any two of them
share at most a single vertex or edge. Let $h_K$ denote the diameter
of $K \in \mathcal{T}_h$ and $\rho_K$ the diameter of the largest
inscribed ball. We assume that $1 \leq \max_{K \in \mathcal{T}_h}
\frac{h_K}{\rho_K} \leq R$ independent of $h$ and that $h =
\max_{K\in \mathcal{T}_h} h_K$. $\mathcal{E}_h$ is going to be the
set of all edges of the triangles in $\mathcal{T}_h$, $D$ the set of
nodes. Associated to $\mathcal{T}_h$ is the space $\Wh \subset \W$
of functions whose restrictions to any $K \in \mathcal{T}_h$ are
linear.

$\pom$ is triangulated by $\{l \in \mathcal{E}_h : l \subset
\pom\}$. $\Whg$ denotes the corresponding space of piecewise linear
functions, and $\Whgs$ the subspace of those supported on $\gs$.
Finally, $\Whmg \subset \Wmg$,
$$\mathcal{A}_h = \A \cap (\Wh \times \Whg)\ $$
and $X^4_h = \Wh \times \Whgs$.

We denote by $i_h: \Wh \hookrightarrow \W$, $j_h :
\Whgs\hookrightarrow \Wgs$ and $k_h: \Whmg \hookrightarrow \Wmg$ the
canonical inclusion maps. A discretization of the
Steklov--Poincar\'e operator is defined as
\[
S_h = \frac12( W+(I-K')k_h(k_h^* V k_h)^{-1} k_h^*(I-K))
\]
from the single resp.~double layer potentials $V$ and $K$ and the
hypersingular integral operator $W$ of the exterior problem. $S_h$
is well--known to be uniformly coercive for small $h$ in the sense
that there exists $h_0>0$ and an $h$--independent $\alpha_S>0$ such
that for all $0<h<h_0$
$$\langle S_h u_h, u_h\rangle \geq \alpha_S \|u_h\|_{\Wg}^2.$$
Furthermore, in this case
\begin{equation}\label{sdisc}
\|(S_h-S)u\|_\Wmg \leq C_S\ \mathrm{dist}_{\Wmg}(V^{-1}(1-K)u,\Whmg)
\end{equation}
for all $u \in \Wg$ and all $0<h<h_0$.

As before, $(\hat u, \hat v)$ denotes a minimizer of $J$ over
$\mathcal{A}$, while $(\hat u_h,\hat v_h)$ minimizes the approximate
functional
$$J_h(u_h, v_h) = \int_\om W^{**}(\nabla u_h) + \frac{1}{2} \langle
S_h(u_h|_\pom +v_h),u_h|_\pom +v_h\rangle - \lambda_h(u_h,v_h),$$
$$\lambda_h(u_h,v_h) = \langle t_0 +S_h u_0, u_h|_\pom + v_h\rangle + \int_\om f
u_h,$$ over $\mathcal{A}_h$. For simplicity, abbreviate the stress
$DW^{**}(\nabla \hat u)$ by $\sigma$ and the indicator $Q(\nabla
\hat u)$ for microstructure by $\xi$. Similarly, write $\sigma_h$
and $\xi_h$ for the corresponding quantities associated to $\hat
u_h$. The following a priori estimate holds.

\begin{thm} \label{apriori}
The Galerkin approximations of the stress $\sigma$, exterior
boundary values $u|_\pom + v$ and the other quantities in Lemma
\ref{uniqueness}
converge for $h \to 0$. \\
a) There is an $h$--independent $C>0$ such that
\begin{eqnarray*}
& & \hspace*{-0.9cm}\|\sigma - \sigma_h\|_{L^{\frac{4}{3}}(\om)}^2 + \|(\hat u-\hat u_h)|_{\pom}+\hat v-\hat v_h\|_{W^{\frac{1}{2},2}(\pom)}^2 \\
& & + \|\mathbb{P}\nabla \hat u-\mathbb{P}\nabla \hat
u_h\|_{L^2(\om)}^2 + \|(\xi+\xi_h)^{1/2} A \nabla(\hat u-\hat
u_h)\|_{L^2(\om)}^2 +\|\xi-\xi_h\|_{L^2(\om)}^2\\ & \leq & C
\inf_{(U_h,V_h) \in \mathcal{A}_h} \big\{ \|u-U_h\|_{W^{1,4}(\om)} +
\|(u-U_h)|_{\pom}+v-V_h)\|_{W^{\frac{1}{2},2}(\pom )}\big\}\\
& & \hspace*{3.0cm}   +\ \mathrm{dist}_{\Wmg}(V^{-1}(1-K)(\hat u +
\hat v - u_0),\Whmg)^2.
\end{eqnarray*}
b) For pure transmission conditions, $\gt = \pom$, the slightly
better estimate
\begin{eqnarray*}
& & \hspace*{-0.9cm}\|\sigma - \sigma_h\|_{L^{\frac{4}{3}}(\om)}^2 +
\|\hat u-\hat u_h\|_{W^{\frac{1}{2},2}(\pom)}^2  +
\|\mathbb{P}\nabla \hat
u-\mathbb{P}\nabla \hat u_h\|_{L^2(\om)}^2 \\
& &+ \|(\xi+\xi_h)^{1/2} A \nabla(\hat u-\hat
u_h)\|_{L^2(\om)}^2 +\|\xi-\xi_h\|_{L^2(\om)}^2\\
&\leq & C \inf_{U_h \in \mathcal{A}_h} \big\{ \|\nabla \hat u -
\nabla U_h\|^2_{L^{4}(\om)} +
\|\hat u-U_h\|^2_{W^{\frac{1}{2},2}(\pom)}\big\} \\
& & \hspace*{0.9cm}  + \ \mathrm{dist}_{\Wmg}(V^{-1}(1-K)(\hat u -
u_0),\Whmg)^2
\end{eqnarray*}
holds.
\end{thm}
\begin{proof}
We integrate (\ref{W3}) and use H\"older's inequality as well as the
uniform bound on the norm of minimizers (the first assertion in
Lemma \ref{uniqueness}) to obtain
\begin{equation} \label{basicsigmaest}
\|\sigma - \sigma_h\|_{L^{\frac{4}{3}}(\om)}^2 \lesssim \int_\om
(\sigma - \sigma_h) \nabla (\hat u - \hat u_h).
\end{equation}
Most of the remaining terms on the left hand side are similarly
bounded with the help of (\ref{W4}):
\begin{align}
\|\mathbb{P}\nabla \hat u-\mathbb{P}\nabla \hat u_h\|_{L^2(\om)}^2 +
\|(\xi+\xi_h)^{1/2} A \nabla(\hat u-\hat u_h)\|_{L^2(\om)}^2 +\|\xi-\xi_h\|_{L^2(\om)}^2 \nonumber\\
 \lesssim \int_\om (\sigma - \sigma_h) \nabla (\hat
u - \hat u_h) \label{basicrestest}.
\end{align}
Adding the inequalities,
\begin{eqnarray*}
LHS^2&:= & \|\sigma - \sigma_h\|_{L^{\frac{4}{3}}(\om)}^2 + \|(\hat u-\hat u_h)|_{\pom}+\hat v-\hat v_h\|_{W^{\frac{1}{2},2}(\pom)}^2 \\
& & + \|\mathbb{P}\nabla \hat u-\mathbb{P}\nabla \hat
u_h\|_{L^2(\om)}^2 +
\|(\xi+\xi_h)^{1/2} A \nabla(\hat u-\hat u_h)\|_{L^2(\om)}^2 +\|\xi-\xi_h\|_{L^2(\om)}^2 \\
 &\lesssim& \textstyle{\int_\om} (\sigma - \sigma_h) (\nabla
\hat u -
\nabla \hat u_h) + \langle S (\hat u-\hat u_h)|_\pom + \hat v-\hat v_h, (\hat u - \hat u_h)|_\pom +\hat v-\hat v_h\rangle\\
& =& -\int_\om \sigma \nabla \hat u_h - \langle S(\hat u|_\pom + \hat v), \hat u_h|_\pom+\hat v_h\rangle \\
& & \qquad - \int_\om \sigma_h \nabla \hat u - \langle S(\hat u_h|_\pom + \hat v_h), \hat u|_\pom+\hat v\rangle\\
& & \qquad +\int_\om \sigma \nabla \hat u + \langle S(\hat u|_\pom + \hat v), \hat u|_\pom+\hat v\rangle \\
& & \qquad + \int_\om \sigma_h \nabla \hat u_h + \langle S_h(\hat
u_h|_\pom + \hat v_h), \hat u_h|_\pom+\hat v_h\rangle\\& & \qquad
+\langle (S-S_h)(\hat u_h|_\pom + \hat v_h), \hat u_h|_\pom+\hat
v_h\rangle.
\end{eqnarray*}
Let $(U,V) \in \mathcal{A}$, $(U_h,V_h) \in \mathcal{A}_h$. Applying
the variational inequality (\ref{VI}) to the third and fourth line
and rearranging terms leads to
\begin{eqnarray*}
LHS^2 &\lesssim & \int_\om \sigma \nabla (U-\hat u_h) + \langle S(\hat u|_\pom + \hat v), U|_\pom+V-\hat u_h|_\pom-\hat v_h\rangle \\
& & \quad + \int_\om \sigma_h \nabla (U_h-\hat u) + \langle S(\hat u_h|_\pom + \hat v_h), U_h|_\pom+V_h-\hat u|_\pom-\hat v\rangle\\
& & \quad + \lambda(\hat u-U, \hat v-V) + \lambda(\hat u_h-U_h, \hat v_h-V_h)\\
& & \quad + \langle (S-S_h) ( \hat u_h|_\pom + \hat v_h - u_0),
(\hat u_h-U_h)|_\pom + \hat v_h-V_h \rangle\\
& = & \int_\om \sigma \nabla (U-\hat u_h) + \langle S(\hat u|_\pom + \hat v), (U-\hat u_h)|_\pom+V-\hat v_h\rangle\\
& & \quad + \int_\om \sigma \nabla (U_h-\hat u) + \langle S(\hat u|_\pom + \hat v), (U_h-\hat u)|_\pom+V_h-\hat v\rangle\\
& & \quad + \int_\om (\sigma_h-\sigma) \nabla (U_h-\hat u) + \langle S((\hat u_h-\hat u)|_\pom + \hat v_h-\hat v), (U_h- \hat u)|_\pom+V_h-\hat v\rangle\\
& & \quad + \lambda(\hat u-U, \hat v-V) + \lambda(\hat u_h-U_h, \hat
v_h-V_h)\\
&&\quad + \langle (S-S_h) (\hat u|_\pom + \hat v - u_0), (\hat
u_h-U_h)|_\pom + \hat v_h-V_h
\rangle\\
&&\quad + \langle (S-S_h) ( (\hat u_h-\hat u)|_\pom + \hat v_h -\hat
v), (\hat u_h-U_h)|_\pom + \hat v_h-V_h \rangle.
\end{eqnarray*}
H\"older's inequality tells us that
$$\int_\om (\sigma_h-\sigma) \nabla (U_h-\hat u) \leq \|\sigma_h-\sigma\|_{L^{\frac{4}{3}}(\om)}\|\nabla (U_h-\hat u)\|_{L^{4}(\om)},$$
and the continuity of $S$ allows to bound
$$\langle S((\hat u_h-\hat
u)|_\pom + \hat v_h-\hat v), (U_h-\hat u)|_\pom+V_h-\hat v\rangle$$
by a multiple of
\begin{eqnarray*}
& & \hspace*{-1.0cm} \|(\hat u_h-\hat u)|_\pom + \hat v_h-\hat
v\|_{W^{\frac{1}{2},2}(\pom)}\|(U_h-\hat
u)|_\pom+V_h-\hat v\|_{W^{\frac{1}{2},2}(\pom)}\\
&\lesssim& \varepsilon \|(\hat u_h-\hat u)|_\pom + \hat v_h-\hat
v\|_{W^{\frac{1}{2},2}(\pom)}^2 + \frac{1}{\varepsilon} \|(U_h-\hat
u)|_\pom+V_h-\hat v\|_{W^{\frac{1}{2},2}(\pom)}^2
\end{eqnarray*}
for small $\varepsilon>0$. Similarly, the last two lines are, up to
prefactors, bounded by
\begin{eqnarray*}
& & \varepsilon \|(\hat u_h-\hat u)|_\pom + \hat v_h-\hat
v\|_{W^{\frac{1}{2},2}(\pom)}^2 +
(1+\frac{1}{\varepsilon})\|(U_h-\hat u)|_\pom+V_h-\hat
v\|^2_{W^{\frac{1}{2},2}(\pom)}\\
&& + \|(S-S_h) (\hat u|_\pom + \hat v -
u_0)\|^2_{W^{\frac{1}{2},2}(\pom)}.
\end{eqnarray*}
Thus, choosing $(U,V)=(\hat u_h,\hat v_h)$,
\begin{eqnarray*}
LHS^2 &\lesssim &  \|\sigma_h-\sigma\|_{L^{\frac{4}{3}}(\om)}\|\nabla (U_h-\hat u)\|_{L^{4}(\om)} \\
& & \, + \|(U_h-\hat u)|_\pom+V_h-\hat
v\|^2_{W^{\frac{1}{2},2}(\pom)} + \|(S-S_h) (\hat u|_\pom + \hat v -
u_0)\|^2_{W^{\frac{1}{2},2}(\pom)}\\
& & \, + \int_\om \sigma \nabla (U_h-\hat u) + \langle S(\hat
u|_\pom + \hat v), (U_h-\hat u)|_\pom+V_h-\hat v\rangle-
\lambda(U_h-\hat u, V_h-\hat v).
\end{eqnarray*}
If $\gt = \pom$, the variational inequality (\ref{VI}) becomes an
equality, the last line vanishes and b) follows. In the general
case, we estimate the last line by
\begin{eqnarray*}
\|\sigma\|_{L^{\frac{4}{3}}(\om)} \|\nabla (U_h-\hat
u)\|_{L^{4}(\om)}+ \|f\|_{L^{\frac{4}{3}}(\om)}\|U_h- \hat
u\|_{L^{4}(\om)} & &\\ + \|S(\hat u|_\pom+\hat
v-u_0)-t_0\|_{W^{-\frac{1}{2},2}(\pom)} \|(U_h-\hat
u)|_\pom+V_h-\hat v\|_{W^{\frac{1}{2},2}(\pom)},& &
\end{eqnarray*}
recalling that
$$\lambda(U_h-\hat u, V_h-\hat v) = \langle t_0 +S u_0, (U_h-\hat u)|_\pom + V_h-\hat v\rangle + \int_\om f (U_h-\hat u).$$
\end{proof}

In particular, we can stably compute the approximate solutions in
the exterior domain from $\hat u_h|_\pom + \hat v_h$.

\vspace*{0.2cm}
\section{Adaptive Grid Refinement}\label{sec::adaptive}

In order to set up an adaptive algorithm, we now establish an a
posteriori estimate of residual type. It allows to localize the
approximation error and leads to an adaptive mesh refinement
strategy. A related and somewhat more involved estimate for the
linear Laplace operator with unilateral Signiorini contact has been
considered in \cite{hn}.

Let $(\hat u, \hat v) \in \A$, $(\hat u_, \hat v_h) \in \A_h$
solutions of the continuous resp.~discretized variational
inequality. We define a simple approximation $(\pi_h \hat u, \pi_h
\hat v) \in \A_h$ of $(\hat u, \hat v)$ as follows: $\pi_h \hat u$
is going to be the Clement interpolant of $\hat u$, and $\pi_h \hat
v= \hat v_h$.

The next Lemma collects the crucial properties of Clement
interpolation (see e.g.~\cite{bs}).

\begin{lem}\label{interpol}
Let $K \in \mathcal{T}_h$ and $E \in \mathcal{E}_h$. Then with
$\omega_K = \bigcup_{\overline{K'}\cap\overline{K} \neq
\emptyset}K'$ and $\omega_E = \bigcup_{\overline{E'}\cap E \neq
\emptyset}E'$ we have:
\begin{eqnarray*}
\|\hat u - \pi_h \hat u\|_{L^4(K)} &\lesssim& h_K \|\hat u\|_{W^{1,4}(\omega_K)}\ ,\\
\|\hat u - \pi_h \hat u\|_{L^2(E)} &\lesssim& h_E^{1/2} \|\hat
u\|_{W^{\frac{1}{2},2}(\omega_E)}.\\
\end{eqnarray*}
\end{lem}

We are going to prove the following a posteriori estimate:

\begin{thm}\label{aposteriori}
\begin{eqnarray*}
& &\hspace{-0.6cm}\|\sigma - \sigma_h\|_{L^{\frac{4}{3}}(\om)}^2 +
\|(\hat u-\hat u_h)|_{\pom}+\hat v-\hat
v_h\|_{W^{\frac{1}{2},2}(\pom)}^2 + \|\mathbb{P}\nabla \hat
u-\mathbb{P}\nabla \hat
u_h\|_{L^2(\om)}^2\\
& & \hspace{-0.6cm} + \|(\xi+\xi_h)^{1/2} A \nabla(\hat u-\hat
u_h)\|_{L^2(\om)}^2 +\|\xi-\xi_h\|_{L^2(\om)}^2\\  &&\quad \lesssim
\,\, \eta_\om + \eta_C + \eta_S +
\mathrm{dist}_{\Wmg}(V^{-1}(1-K)(\hat u_h|_\pom + \hat
v_h-u_0),\Whmg)^2 \ ,
\end{eqnarray*}
where
\begin{eqnarray*}
\eta_\om& =&\sum_{K} h_K \|f\|_{L^{4/3}(K)}+\sum_{E \cap \pom =
\emptyset} h_E \|[\nu_E\cdot \sigma_h]\|_{L^2(E)}\ ,\\
\eta_C & = & \eta_{C,1}+\eta_{C,2}=\sum_{E \subset \gs }
\|(\nu_E\cdot \sigma_h)_+\|_{L^2(E)} + \sum_{E \subset \gs} \int_E
(\nu_E \cdot \sigma_h )_-\ \hat v_h\ ,\\ \eta_S& =& \sum_{E \subset
\pom}h_E^{1/2}\ \| S_h( \hat u_h|_\pom + \hat
v_h-u_0)+(\nu_\pom\cdot \sigma_h)-t_0\|_{L^2(E)}\ .
\end{eqnarray*}
\end{thm}
\begin{rem}
a) Also the constant prefactors, suppressed in our notation $\lesssim$, are explicitly known.\\
b) The main point of this estimate is to show that the a posteriori
estimates for the contact part (\cite{hn}) and the double--well term
(\cite{cp1}) are compatible. More sophisticated bounds related to a
different choice of $\pi_h$ generalize to our setting in a similar
way. As a simple case, a more considerate (sign--preserving) choice
of $\pi_h \hat v$ with $\|\hat v - \pi_h \hat v \|_{L^2(\gs)}
\lesssim h^\alpha\|\hat v\|_{L^2(\gs)}$ could be used to gain an
$h^\alpha$ in $\eta_{C,1}$ at the expense of modifying
$$\eta_{C,2} = \sum_{E \subset \gs} \int_E (\nu_E \cdot \sigma_h
)_-\ \pi_h^1\hat v_h\ ,$$ as long as we only assure that
$\int_E(\pi_h\hat v_h - \hat v) \leq \int_E \pi_h^1\hat v_h$ for
some auxiliary interpolation operator $\pi_h^1$ (see
e.g.~\cite{hn}).\\
c) As in \cite{scalar}, it is straight forward to introduce an
additional variable on the boundary to obtain estimates that do not
involve the incomputable difference $S_h-S$. Similarly, we might
also use the formulation of Bartels \cite{b} with explicit Young
measures in the interior part.
\end{rem}

\begin{proof}[Proof of Theorem \ref{aposteriori}] As in the proof of Theorem \ref{apriori}, we start with the
inequality
\begin{eqnarray*}
& &  \hspace*{-0.9cm}LHS^2:=\|\sigma - \sigma_h\|_{L^{\frac{4}{3}}(\om)}^2 + \|(\hat u-\hat u_h)|_{\pom}+\hat v-\hat v_h\|_{W^{\frac{1}{2},2}(\pom)}^2 \\
& & + \|\mathbb{P}\nabla \hat u-\mathbb{P}\nabla \hat
u_h\|_{L^2(\om)}^2 + \|(\xi+\xi_h)^{1/2} A \nabla(\hat u-\hat
u_h)\|_{L^2(\om)}^2 +\|\xi-\xi_h\|_{L^2(\om)}^2\\ & \lesssim &
\int_\om (\sigma - \sigma_h) \nabla (\hat u-\hat u_h) + \langle
S((\hat u-\hat u_h)|_\pom + \hat v-\hat v_h), (\hat u-\hat
u_h)|_\pom + \hat v-\hat v_h\rangle .
\end{eqnarray*}
Using the variational inequality and its discretized variant results
in
\begin{eqnarray*}
LHS^2 &\lesssim&  \lambda(\hat u-\hat u_h, \hat v-\hat v_h) -
\int_\om \sigma_h \nabla(\hat u-\hat u_h)-\langle S(\hat u_h|_\pom +
\hat
v_h,(\hat u-\hat u_h)|_\pom + \hat v-\hat v_h \rangle \\
&=&  \lambda(\hat u-\hat u_h, \hat v-\hat v_h) - \int_\om \sigma_h
\nabla(\hat u-\hat u_h)-\langle S_h(\hat u_h|_\pom + \hat
v_h,(\hat u-\hat u_h)|_\pom + \hat v-\hat v_h \rangle \\
& & \quad + \langle (S_h-S)(\hat u_h|_\pom + \hat v_h,(\hat u-\hat
u_h)|_\pom + \hat v-\hat v_h \rangle\\
&\leq& \lambda(\hat u-u_h, \hat v-v_h)-\int_\om \sigma_h\nabla(\hat
u -u_h) -\langle S_h(\hat u_h|_\pom + \hat v_h, (\hat u- u_h)|_\pom
+ \hat v-v_h \rangle\\& & \qquad +\ \langle (S_h-S)(\hat u_h|_\pom +
\hat v_h-u_0), (\hat u-\hat u_h)|_\pom + \hat v-\hat v_h \rangle\\
&=&  \int_\om f(\hat u - u_h) - \sum_{E\cap \pom = \emptyset} \int_E
[\nu_E\cdot \sigma_h] (\hat u - u_h)\\& &  \quad -\ \langle S_h(
\hat u_h|_\pom + \hat v_h-u_0)+(\nu_\pom\cdot \sigma_h)-t_0), (\hat
u -
u_h)|_\pom + \hat v - v_h\rangle\\
& & \quad +\ \int_\gs (\nu_\pom \cdot \sigma_h )\ (\hat v - v_h)\\
& & \quad +\langle (S_h-S)(\hat u_h|_\pom + \hat v_h - u_0), (\hat
u-\hat u_h)|_\pom + \hat v-\hat v_h \rangle
\end{eqnarray*}
for all $(u_h, v_h) \in \A_h$. Here, $\nu_E$ and $\nu_\pom$ denote
the outward--pointing unit normal vector to an edge $E \subset
\overline{K}$, resp.~to $\pom$, and $[\nu_E\cdot \sigma_h]$ is the
jump of the discretized normal stress across $E$. According to
estimate (\ref{sdisc}) for $S_h-S$ and Young's inequality, the last
term contributes the not explicitly computable
$\mathrm{dist}_{\Wmg}(V^{-1}(1-K)(\hat u_h|_\pom + \hat
v_h),\Whmg)^2$.

We are going to choose $(u_h, v_h) = (\pi_h\hat u, \pi_h \hat v)$.
Then, the first three terms on the right hand side can be estimated
with the help of Lemma \ref{interpol} and H\"{o}lder's inequality:
$$\int_\om f(\hat u - \pi_h\hat u) \leq \|\hat u\|_{\W}\ \Big(\sum_{K} h_K^{4/3} \int_K |f|^{4/3}\Big)^{3/4}\ ,$$
$$\Big|\sum_{E \cap \pom = \emptyset}\int_E [\nu_E\cdot \sigma_h] (\hat u - \pi_h\hat u)\Big| \leq \|\hat u\|_{\Wg}
\ \Big( \sum_{E \cap \pom = \emptyset} h_E \int_E |[\nu_E\cdot
\sigma_h]|^{2}\Big)^{1/2}$$ and
\begin{eqnarray*}
& &\hspace{-1.0cm}|\langle S_h( \hat u_h|_\pom + \hat
v_h-u_0)+(\nu_\pom\cdot \sigma_h)-t_0), (\hat u - \pi_h\hat u)|_\pom
+ \hat v - \pi_h\hat v\rangle|
\\& &\hspace{-1.0cm} \leq (\| \hat u\|_{\Wg} + \|\hat v\|_{\Wg} + \|\hat v_h\|_{\Wg})\ \| S_h( \hat u_h|_\pom +
\hat v_h-u_0)+(\nu_\pom\cdot \sigma_h)-t_0\|_{\Wmg}\ .
\end{eqnarray*}

 Note that the boundedness of the
set of minimizers, Lemma \ref{uniqueness}, provides an explicit
uniform bound on both $\|\hat u, \hat v\|_{X}$ and $\|\hat u_h, \hat
v_h\|_{X}$ in terms of the norms of the data. The $\Wmg$--norm leads
to $\eta_S$ (\cite{cast}).

The remaining term requires a slightly more precise analysis.
Decompose
$$(\nu_\pom \cdot \sigma_h ) = (\nu_\pom \cdot \sigma_h )_+- (\nu_\pom \cdot \sigma_h )_-$$
into its positive and negative parts. For a classical exact
solution, the Signiorini condition requires $(\nu_\pom \cdot
\sigma)_+=0$, and we estimate the corresponding term as above:
$$\Big|\int_\gs (\nu_\pom \cdot \sigma_h )_+\ (\hat v - \pi_h \hat v)\Big|
\lesssim (\|\hat v\|_{\Wg}+\|\hat v_h\|_{\Wg})\ \Big( \int_\gs
|(\nu_E\cdot \sigma_h)_+|^{2}\Big)^{1/2}\ .$$ For the negative part,
we may drop the unknown term:
\begin{eqnarray*}
-\int_\gs (\nu_\pom \cdot \sigma_h )_-\ (\hat v - v_h)& =&\sum_{E
\subset \gs}(\nu_E \cdot \sigma_h )_- \int_E(v_h-\hat v)\\
&\leq& \sum_{E \subset \gs}(\nu_E \cdot \sigma_h )_- \int_E \hat
v_h.
\end{eqnarray*}
The a posteriori estimate follows.
\end{proof}

\end{document}